\documentclass[11pt]{amsart}

\usepackage{amscd}
\usepackage{amsmath, amssymb}
\usepackage{amsfonts}
\newcommand{\de}{\partial}

\newcommand{\ddbar}{\sqrt{-1} \partial \overline{\partial}}

\newcommand{\ov}[1]{\overline{#1}}

\newcommand{\ti}[1]{\tilde{#1}}
\newcommand{\vp}{\varphi}

\renewcommand{\leq}{\leqslant}
\renewcommand{\geq}{\geqslant}

\renewcommand{\ge}{\geqslant}

\newcommand{\be}{\begin{equation}}
\newcommand{\ee}{\end{equation}}

\begin{document}
\newtheorem{claim}{Claim}
\newtheorem{theorem}{Theorem}[section]
\newtheorem{lemma}[theorem]{Lemma}
\newtheorem{corollary}[theorem]{Corollary}
\newtheorem{proposition}[theorem]{Proposition}
\newtheorem{question}{question}[section]
\theoremstyle{definition}
\newtheorem{remark}[theorem]{Remark}

\numberwithin{equation}{section}

\title[Regularity of envelopes]{Regularity of envelopes in K\"ahler classes}
\author[V. Tosatti]{Valentino Tosatti}
\address{Department of Mathematics, Northwestern University, 2033 Sheridan Road, Evanston, IL 60208}
\email{tosatti@math.northwestern.edu}

\begin{abstract}We establish the $C^{1,1}$ regularity of quasi-psh envelopes in a K\"ahler class, confirming a conjecture of Berman.
\end{abstract}

\maketitle

\section{Introduction}
Let $(X^n,\omega)$ be a compact K\"ahler manifold, and $\theta=\omega+\ddbar v$ a closed real $(1,1)$-form cohomologous to $\omega$, where $v\in C^\infty(X,\mathbb{R})$. The envelope (or extremal function) $u_\theta$ is defined by
\[\begin{split}u_\theta(x)&=\sup\{u(x)\ |\ u\in PSH(X,\theta), u\leq 0\}\\
&=-v+\sup\{u(x)\ |\ u\in PSH(X,\omega), u\leq v\},\end{split}\]
is a $\theta$-psh function with minimal singularities in the class $[\omega]$, and has received much attention recently (see for example \cite{BB,BBGZ,WN} and references therein). By Berman-Demailly \cite{BD}, we know that the complex Hessian (or equivalently the Laplacian) of $u_\theta$ belongs to $L^\infty(X)$, and so in particular $u_\theta$ is $C^{1,\alpha}(X)$ for all $0<\alpha<1$. A direct PDE proof was given by Berman in \cite{Be2}.

Here we establish the optimal regularity result for the envelope, which was previously only known when $[\omega]\in H^2(X,\mathbb{Q})$ by \cite{Be3} (see also \cite{Be}). This resolves affirmatively a conjecture of Berman \cite[Conjecture 1.10]{Be3}:
\begin{theorem}\label{main}
The envelope $u_\theta$ is in $C^{1,1}(X)$.
\end{theorem}
This is in general optimal, see e.g. \cite[Example 5.2]{Be3} for examples on toric manifolds.

In fact, combining our result with the arguments in \cite[Proof of Theorem 2.5]{DR}, we obtain the same $C^{1,1}$ regularity result for the ``rooftop envelopes''
$$P(v_1,\dots,v_k)(x)=\sup\{u(x)\ |\ u\in PSH(X,\omega), u\leq \min_{j=1,\dots,k} v_j\},$$
where the $v_j$'s are $C^{1,1}$ functions, see Theorem \ref{main2} below.

Also, using Theorem \ref{main} together with the arguments in \cite[Theorem 3.4]{Be3}, we obtain a shlightly shorter proof of the identity
\begin{equation}\label{form}
(\theta+\ddbar u_\theta)^n=\chi_{\{u_\theta=0\}} \theta^n,
\end{equation}
which clearly implies
\begin{equation}\label{form2}
\int_{\{u_\theta=0\}}\theta^n=\int_X\omega^n,
\end{equation}
and which was proved (in more generality) in \cite[Corollary 2.5]{BD}. Indeed it is classical that the Monge-Amp\`ere operator $(\theta+\ddbar u_\theta)^n$ vanishes outside the contact set $\{u_\theta=0\}$ (see e.g. \cite[Proposition 3.1]{Be3} or \cite[Proposition 2.10]{BB}), and by Theorem \ref{main} we know that $\nabla_i u_\theta$ is Lipschitz (for any $1\leq i\leq n$, working in a local coordinate chart) and so $\nabla \nabla_i u_\theta=0$ a.e. on the set $\{\nabla_i u_\theta=0\}$ (see e.g. \cite[Theorem 3.2.6]{AT}), which contains the contact set. Therefore a.e. on the contact set we have $\nabla^2u_\theta=0$ and so $\theta+\ddbar u_\theta=\theta$, which proves \eqref{form}.

The proof of the Theorem \ref{main}, which is given in section \ref{sectmain}, is obtained by using Berman's result \cite{Be2} that the envelope $u_\theta$ is in fact the limit of solutions of a $1$-parameter family of complex Monge-Amp\`ere equations, together with the technique recently introduced by Chu, Weinkove and the author \cite{CTW,CTW2} to obtain uniform $C^{1,1}$ estimates for such equations.
A generalization of this result to ``rooftop envelopes'' (in the sense of \cite{DR}) is proved in section \ref{sectenv}.\\

After the first version of this paper was posted on the arXiv, we were informed that J. Chu and B. Zhou independently proved Theorem \ref{main} in \cite{CZ}.\\

\noindent
{\bf Acknowledgments. }The author is grateful to M.P\u{a}un, D.Witt Nystr\"om, J.Xiao for useful discussions, to J.Chu and B.Weinkove for our collaboration, and to the referee for useful comments.  These results were obtained during the author's visit at the Center for Mathematical Sciences and Applications at Harvard University, which he would like to thank for the hospitality. The author was partially supported by NSF grant DMS-1610278.

\section{$C^{1,1}$ regularity of envelopes}\label{sectmain}
In this section we give the proof of Theorem \ref{main}. Following the approach of \cite{Be2}, we consider the family of complex Monge-Amp\`ere equations
\begin{equation}\label{ma2}
(\theta+\ddbar u_\beta)^n=e^{\beta u_\beta}\omega^n,
\end{equation}
where $\beta\in\mathbb{R}_{\geq 0}$, the function $u_\beta$ is smooth and $\theta+\ddbar u_\beta$ is a K\"ahler metric on $X$. This is solvable thanks to the work of Aubin \cite{A} and Yau \cite{Ya}. Recall also that we write $\theta=\omega+\ddbar v$ for a smooth function $v$.

Berman shows in \cite{Be2} that
\begin{equation}\label{berm}
|u_\beta|\leq C, \quad |\Delta_g u_\beta|\leq C, \quad |u_\beta-u_\theta|\leq C\frac{\log\beta}{\beta},
\end{equation}
for a uniform constant $C$ independent of $\beta$ (and which depends only on the $C^{1,1}$ norm of $v$), from which it follows that $u_\beta$ converges to $u_\theta$ in $C^{1,\alpha}(X)$ for any $0<\alpha<1$, as $\beta\to\infty$.

Our main result is that for all $\beta\in\mathbb{R}_{\geq 0}$ we have
\begin{equation}\label{goal}
|\nabla^2 u_\beta|_g\leq C,
\end{equation}
for a uniform $C$, which immediately implies Theorem \ref{main}. As will be apparent from the proof, the constant $C$ depends only on the $C^{1,1}$ norm of $v$.
Let $\vp=u_\beta+v$ and rewrite \eqref{ma2} as
\begin{equation}\label{ma}
(\omega+\ddbar \vp)^n=e^{\beta (\vp-v)}\omega^n,
\end{equation}
where $\ti{\omega}:=\omega+\ddbar \vp$ is a K\"ahler metric,
and the idea is to follow very closely the method introduced by Chu, Weinkove and the author in \cite{CTW,CTW2}. We thus let $\lambda_1(\nabla^2\vp)$ be the largest eigenvalue of $\nabla^2\vp$ with respect to $g$, and the goal is to prove that $\lambda_1(\nabla^2\vp)\leq C$ for a uniform constant $C$. Indeed, once we prove this, since the trace of $\nabla^2\vp$ is $\Delta_g\vp$ which is bounded below by $-n$, we will conclude that
$$|\nabla^2 \vp|_g\leq C,$$
which implies \eqref{goal}. To this end, we apply the maximum principle to
$$Q = \log \lambda_1( \nabla^2 \varphi) + h(| \partial \varphi |^2_g) -A \varphi,$$
(defined on the set where $\lambda_1(\nabla^2\vp)>0$, which we may assume is nonempty) where $A>0$ is a uniform constant to be determined and
\begin{equation} \label{defh}
h(s) = - \frac{\lambda}{2} \log (1+ \sup_M |\partial \varphi|^2_g  - s),
\end{equation}
where $\lambda=(1+2\sup_X |\de v|^2_g)^{-1}\leq 1$, is a small uniform constant. The only difference between this quantity and the corresponding one in \cite{CTW2} is that there we just took $\lambda=1$.
We have
\begin{equation} \label{proh}
\frac{\lambda}{2}\geq h' \geq \frac{\lambda}{2+2\sup_M |\partial \varphi|^2_g}>0, \quad \textrm{and } h'' = \frac{2}{\lambda} (h')^2\geq 2(h')^2,
\end{equation}
where we are evaluating $h$ and its derivatives at $|\de\vp|^2_g$. The bounds \eqref{berm} show that the the last two terms in $Q$ are uniformly bounded.

We work at a point $x_0$ where the maximum is achieved, and as in \cite{CTW2} we choose local normal coordinates for $g$ near $x_0$, so that $(\ti{g}_{i\ov{j}})(x_0)$ is diagonal, as well as constant vector fields $\{V_\alpha\}$ near $x_0$ which at that point form an orthonormal basis of eigenvectors of $\nabla^2\vp$, with $\nabla^2\vp(V_1,V_1)(x_0)=\lambda_1$.

We also apply the same perturbation argument as in \cite{CTW2}, so that $Q$ gets replaced by the local quantity $\hat{Q}$ defined near $x_0$ as in \cite{CTW2} by
$$\hat{Q} = \log \lambda_1( \Phi) + h(| \partial \varphi |^2_g) -A \varphi,$$
where $\Phi$ is the endomorphism of $TX$ given by
$$\Phi^\mu_\nu=g^{\mu\gamma}(\nabla^2_{\gamma\nu}\vp-\delta_{\gamma\nu}+V_1^\gamma V_1^\nu),$$
where $(V_1^\nu)$ are the components of $V_1$. The largest eigenvalue $\lambda_1( \Phi)$ now varies smoothly near $x_0$ and $\hat{Q}$ achieves a local maximum at that point.
Writing $\lambda_\alpha=\lambda_\alpha(\Phi)$, the goal is to show that $\lambda_1(x_0)\leq C$, for a uniform constant $C$.
We claim that at $x_0$ we have
\begin{equation} \label{e4}
\begin{split}
0 \ge \Delta_{\ti{g}} \hat{Q}
\ge {} & 2 \sum_{\alpha >1}  \frac{\tilde{g}^{i\ov{i}} |\partial_i (\varphi_{V_{\alpha} V_1})|^2}{\lambda_1(\lambda_1-\lambda_{\alpha})} + \frac{\tilde{g}^{p\ov{p}} \tilde{g}^{q\ov{q}} | V_1(\tilde{g}_{p\ov{q}})|^2}{\lambda_1} - \frac{\tilde{g}^{i\ov{i}} | \partial_i (\varphi_{V_1 V_1})|^2}{\lambda_1^2} \\
{} & + h' \sum_k \tilde{g}^{i\ov{i}} (| \varphi_{ik}|^2 + |\varphi_{i\ov{k}}|^2) +\beta h' |\de\vp|^2_g+ h'' \tilde{g}^{i\ov{i}} |\partial_i | \partial \varphi|^2_g|^2 \\ {} &+ (A-C) \sum_i \tilde{g}^{i\ov{i}}- An+\frac{\beta}{4}\\
{} &\geq 2 \sum_{\alpha >1}  \frac{\tilde{g}^{i\ov{i}} |\partial_i (\varphi_{V_{\alpha} V_1})|^2}{\lambda_1(\lambda_1-\lambda_{\alpha})} + \frac{\tilde{g}^{p\ov{p}} \tilde{g}^{q\ov{q}} | V_1(\tilde{g}_{p\ov{q}})|^2}{\lambda_1} - \frac{\tilde{g}^{i\ov{i}} | \partial_i (\varphi_{V_1 V_1})|^2}{\lambda_1^2} \\
{} & + h' \sum_k \tilde{g}^{i\ov{i}} (| \varphi_{ik}|^2 + |\varphi_{i\ov{k}}|^2) + h'' \tilde{g}^{i\ov{i}} |\partial_i | \partial \varphi|^2_g|^2 \\ {} &+ (A-C) \sum_i \tilde{g}^{i\ov{i}}
 - An.
\end{split}
\end{equation}
Indeed, as in \cite[(2.7)]{CTW2} we have
\begin{equation} \label{e3}
\begin{split} \Delta_{\ti{g}}\hat{Q}  = {} & \frac{\Delta_{\ti{g}}(\lambda_1)}{\lambda_1} - \frac{\tilde{g}^{i\ov{i}} | \partial_i (\varphi_{V_1 V_1})|^2}{\lambda_1^2} + h'\Delta_{\ti{g}} (| \partial \varphi|^2_g) + h'' \tilde{g}^{i\ov{i}} |\partial_i | \partial \varphi|^2_g|^2 \\ & + A \sum_i \tilde{g}^{i\ov{i}}- An,
\end{split}
\end{equation}
and as in \cite[(2.8)]{CTW2}
\begin{equation} \label{e1}
\begin{split}
\Delta_{\ti{g}} (\lambda_1)
\ge {} & 2 \sum_{\alpha >1} \tilde{g}^{i\ov{i}} \frac{ |\partial_i (\varphi_{V_{\alpha} V_1})|^2}{\lambda_1-\lambda_{\alpha}} + \tilde{g}^{i\ov{i}} V_1 V_1 (\tilde{g}_{i\ov{i}})
 - C \lambda_1 \sum_i \tilde{g}^{i\ov{i}}.
\end{split}
\end{equation}
The Monge-Amp\`ere equation \eqref{ma} in local coordinates reads
\begin{equation}\label{logeq}
\log\det\ti{g}=\log\det g+\beta\vp -\beta v,
\end{equation}
and so applying $V_1V_1$ to this and evaluating at $x_0$ we obtain
\begin{equation} \label{e0}
\begin{split}
\tilde{g}^{i\ov{i}} V_1 V_1 (\tilde{g}_{i\ov{i}}) &= \tilde{g}^{p\ov{p}} \tilde{g}^{q\ov{q}} | V_1(\tilde{g}_{p\ov{q}})|^2 + V_1V_1(\log\det g)+\beta V_1 V_1(\vp)
-\beta V_1 V_1(v)\\
&= \tilde{g}^{p\ov{p}} \tilde{g}^{q\ov{q}} | V_1(\tilde{g}_{p\ov{q}})|^2 + V_1V_1(\log\det g)+\beta\lambda_1-\beta V_1V_1(v)\\
&\geq \tilde{g}^{p\ov{p}} \tilde{g}^{q\ov{q}} | V_1(\tilde{g}_{p\ov{q}})|^2 + V_1V_1(\log\det g)+\beta(\lambda_1-C)\\
&\geq \tilde{g}^{p\ov{p}} \tilde{g}^{q\ov{q}} | V_1(\tilde{g}_{p\ov{q}})|^2 + V_1V_1(\log\det g)+\frac{\beta}{2}\lambda_1,\\
\end{split}
\end{equation}
since we may assume that at $x_0$ the largest eigenvalue $\lambda_1$ is large. This gives
\begin{equation} \label{e2}
\begin{split}
\Delta_{\ti{g}} (\lambda_1)  \ge {} & 2 \sum_{\alpha>1} \tilde{g}^{i\ov{i}} \frac{ |\partial _i (\varphi_{V_{\alpha} V_1})|^2}{\lambda_1 - \lambda_{\alpha}} + \tilde{g}^{p\ov{p}} \tilde{g}^{q\ov{q}} |V_1(\tilde{g}_{p\ov{q}})|^2 - C \lambda_1 \sum_i \tilde{g}^{i\ov{i}}+\frac{\beta}{2}\lambda_1.
\end{split}
\end{equation}

Next,  at $x_{0}$,
\begin{equation}\label{e-1}
\begin{split}
\Delta_{\ti{g}}(|\partial\varphi|_{g}^{2}) ={} &\sum_{k}\tilde{g}^{i\overline{i}}(|\varphi_{ik}|^{2}+| \varphi_{i\ov{k}}|^{2})
+2\beta\textrm{Re}\left(\sum_{k}\varphi_{k}(\vp-v)_{\overline{k}}\right)\\
{} & +\tilde{g}^{i\overline{i}}\partial_{i}\partial_{\ov{i}}(g^{k\overline{\ell}})\varphi_{k}\varphi_{\overline{\ell}}\\
\geq {} & \sum_{k}\tilde{g}^{i\overline{i}}(|\varphi_{ik}|^{2}+|\varphi_{i\ov{k}}|^{2})- C\sum_{i}\tilde{g}^{i\overline{i}}+2\beta|\de\vp|^2_g
-2\beta\textrm{Re}\left(\sum_{k}\varphi_{k}v_{\overline{k}}\right)\\
\geq {} & \sum_{k}\tilde{g}^{i\overline{i}}(|\varphi_{ik}|^{2}+|\varphi_{i\ov{k}}|^{2})- C\sum_{i}\tilde{g}^{i\overline{i}}+\beta|\de\vp|^2_g
-\beta|\de v|^2_g,
\end{split}
\end{equation}
where to derive the first line we have applied $\partial_{\ov{k}}$ to  \eqref{logeq}.
But then $$\beta h' |\de v|^2_g\leq \frac{\beta \lambda}{2}|\de v|^2_g\leq \frac{\beta \sup_X |\de v|^2_g}{2+4\sup_X|\de v|^2_g}\leq\frac{\beta}{4},$$
and so combining this with \eqref{e3}, \eqref{e2} and \eqref{e-1}, we see that \eqref{e4} holds.

Now the rest of the proof proceeds exactly as in \cite{CTW2}, since \eqref{e4} is the same as \cite[(2.6)]{CTW2}, and the specific form of the PDE \eqref{ma} is not used anymore in \cite{CTW2} after that point. At a couple of places we used that $h''=2(h')^2$, but in fact the inequality $h''\geq 2(h')^2$ is enough, and this holds in our case. The constant $A$ is chosen at the end of the argument of \cite[Proof of Theorem 1.2]{CTW2}, and it equals $A=C+3,$ where $C$ is the uniform constant in \eqref{e4}. This completes the proof of Theorem \ref{main}.

\section{Rooftop envelopes}\label{sectenv}
In this section we consider a generalization of Theorem \ref{main}, as follows.

Suppose we are now given $C^{1,1}$ functions $v_j,j=1,\dots,k$ on a compact K\"ahler manifold $(X,\omega)$, and we consider the ``rooftop envelope''
$$P(v_1,\dots,v_k)(x)=\sup\{u(x)\ |\ u\in PSH(X,\omega), u\leq \min_{j=1,\dots,k} v_j\}.$$
When $k=1$ this is essentially the same as the envelope we considered in Theorem \ref{main}, but with a weaker regularity assumption. Darvas-Rubinstein proved in \cite{DR} that $P(v_1,\dots,v_k)$ has bounded Laplacian on $X$, in particular it is in $C^{1,\alpha}(X)$ for all $0<\alpha<1$, and that if $[\omega]\in H^2(X,\mathbb{Q})$ then $P(v_1,\dots,v_k)$ is in $C^{1,1}(X)$. This last point used the regularity results of Berman and Demailly \cite{Be3,BD}, and another proof was also given by Berman \cite{Be}. Using Theorem \ref{main}, we can prove the $C^{1,1}$ regularity of $P(v_1,\dots,v_k)$ in general K\"ahler classes:

\begin{theorem}\label{main2}
The rooftop envelope $P(v_1,\dots,v_k)$ is in $C^{1,1}(X)$.
\end{theorem}
\begin{proof}
The argument in \cite[Proof of Theorem 2.5]{DR} reduces this result to proving the case when $k=1$. So we have a function $v\in C^{1,1}(X)$, and consider the envelope
$$P(v)(x)=\sup\{u(x)\ |\ u\in PSH(X,\omega), u\leq v\},$$
and the goal is to show that $P(v)$ is also in $C^{1,1}(X)$.

By using convolution in local charts and gluing them with a partition of unity (see e.g. the appendix in \cite{Ma}) can choose a sequence $v_j$ of smooth functions which converge to $v$ in $C^{1,\alpha}(X)$ for some fixed $0<\alpha<1$, and such that $\|v_j\|_{C^{1,1}(X,g)}\leq C$ for all $j$. For each $j$ and $\beta\geq 0$ solve
\begin{equation}\label{ma3}
(\omega+\ddbar \vp)^n=e^{\beta (\vp-v_j)}\omega^n,
\end{equation}
where $\vp=\vp_{j,\beta}$ and $\omega+\ddbar\vp>0$. As mentioned earlier, Berman \cite{Be2} proved that
\begin{equation}\label{berm2}
|\vp|\leq C, \quad |\Delta_g \vp|\leq C, \quad |\vp-P(v_j)|\leq C\frac{\log\beta}{\beta},
\end{equation}
for a uniform constant $C$ independent of $j,\beta$, from which it follows that for any $j$ fixed $\vp$ converges to $P(v_j)$ in $C^{1,\alpha}(X)$ for any $0<\alpha<1$, as $\beta\to\infty$.
From Theorem \ref{main} and its proof, we also have that
$$|\nabla^2\vp|_g\leq C,$$
independent of $j,\beta$. Therefore $\|P(v_j)\|_{C^{1,1}(X,g)}\leq C$ for all $j$. On the other hand we have that $P(v_j)\to P(v)$ uniformly as $j\to\infty$, which follows easily from the definition, and so we conclude that $P(v)\in C^{1,1}(X)$ as well.
\end{proof}

\end{document}